\documentclass[12pt, twosided]{article}
\usepackage{latexsym,enumerate}
\usepackage{amsmath,amsthm,amsopn,amstext,amscd,amsfonts,amssymb,fontenc,bbm}
\usepackage{verbatim}
\usepackage{graphicx}
\usepackage{txfonts,pxfonts,wasysym}

\usepackage{lineno} %
\usepackage{textcomp}

\newtheorem{theorem}{Theorem}[section]
\newtheorem{corollary}[theorem]{Corollary}

\newtheorem{proposition}[theorem]{Proposition}

\newtheorem{lemma}[theorem]{Lemma}

\newtheorem{innercustomthm}{Theorem}
\newenvironment{customthm}[1]
  {\renewcommand\theinnercustomthm{#1}\innercustomthm}
  {\endinnercustomthm}

\theoremstyle{definition}

\renewcommand{\d}{\delta}
\newcommand{\D}{\Delta}

\begin{document}

\title{Regularity and Planarity of Token Graphs}

\author{
Walter Carballosa\thanks{CONACYT Research Fellow, M\'exico.
Email: waltercarb@gmail.com} \footnotemark[4]\and
Ruy Fabila-Monroy\thanks{Departamento de Matem\'aticas, CINVESTAV, Mexico,
Email: ruyfabila@math.cinvestav.edu.mx.}
\and
Jes\'us Lea\~nos\thanks{Faculty of Natural Sciences and Mathematics, University of Maribor, Slovenia.
Email: jesus.leanos@gmail.com}
\and
Luis Manuel Rivera\thanks{Unidad Acad\'emica de Matem\'aticas, Universidad Aut\'onoma de Zacatecas, Zac., Mexico.
Email: luismanuel.rivera@gmail.com.}
}
\date{}

\maketitle

\begin{abstract}
Let $G=(V,E)$ be a graph of order $n$ and let $1\le k< n$ be an integer. The \emph{$k$-token graph}  of $G$ is  the
graph whose vertices are all the $k$-subsets of $V$, two of which are adjacent whenever their
symmetric difference is a pair of adjacent vertices in $G$.
In this paper we characterize precisely, for each value of $k$, which graphs have a regular $k$-token graph and which connected graphs have a planar $k$-token graph.
\end{abstract}



\section{Introduction.}
Throughout this paper, $G=(V,E)$ denotes a simple graph  of $n$ vertices and $k$ is an integer
with $1 \le k < n$. The \emph{$k$-token graph} $F_k(G)$ of $G$ is  the
graph whose vertices are all the $k$-subsets of $V$, where two such subsets are adjacent whenever their
symmetric difference is a pair of adjacent vertices in $G$. The token graph was
introduced in \cite{FFHH} where some of
their properties were studied. In that paper the authors noted that:
\begin{quote}
``Thus
vertices of $F_k(G)$ correspond to configurations of $k$
indistinguishable tokens placed at different vertices of $G$, where
two configurations are adjacent whenever one configuration can be
reached from the other by moving one token along an edge from its
current position to an unoccupied vertex. ''
\end{quote}
As an example, the $2$-token graph of the cycle graph $C_5$ is shown in Figure \ref{fig:C5}.
Clearly, $F_1(G)\simeq F_{n-1}(G)\simeq G$; we say that $F_1(G)$ and $F_{n-1}(G)$ are the {\it trivial token graphs} of $G$.

\begin{figure}
  \begin{center}
   \includegraphics[width=0.8\textwidth]{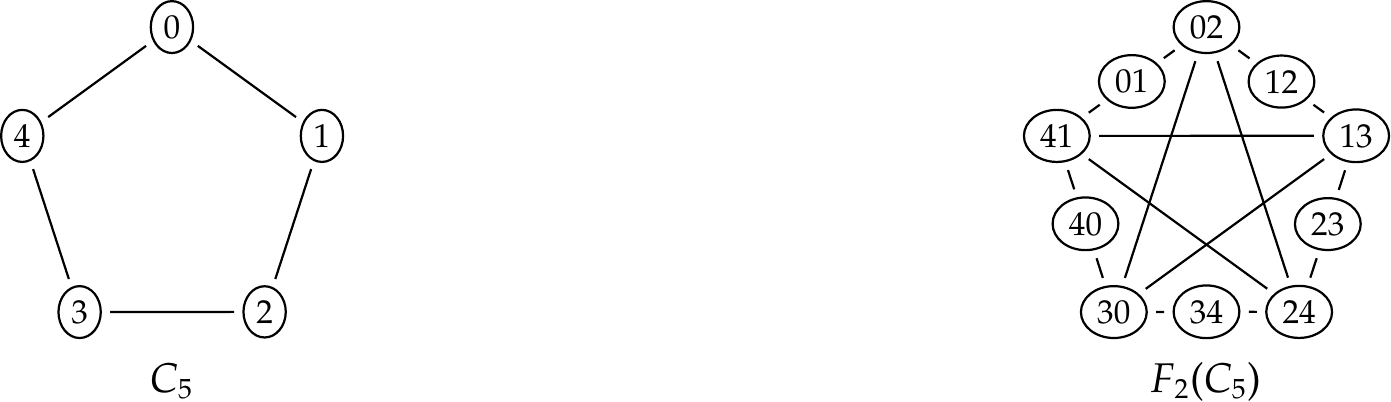}
\end{center}
\caption{The cycle graph $C_5$ and its token graph $F_2(C_5)$.} \label{fig:C5}
\end{figure}


The \emph{Johnson graph} $J(n,k)$ is the graph whose
vertices are the $k$-subsets of an $n$-set, where two such subsets $A$
and $B$ are adjacent whenever $|A \cap B| = k - 1$. Thus,
the Johnson graph $J(n,k)$ is isomorphic to the $k$-token
graph of the complete graph $K_n$, \emph{i.e.},
$J(n,k)\simeq F_k(K_n)$. Therefore, results obtained for token
graphs also apply for Johnson graphs; Johnson graphs are widely studied
due to connections with coding theory, see, \emph{e.g.},
\cite{EB,GS,GWL, lieb, neun}.

We write  $u\sim v$ whenever $u$ and $v$ are adjacent vertices in $G$.
The edge joining these vertices is denoted by $[u,v]$.
For a nonempty set $X\subseteq V$, and a vertex $v\in V$, $N_X(v)$ denotes
the set of neighbors that $v$ has in $X$, \emph{i.e.} $N_X(v):=\{u\in X: u\sim v\}$;
the degree of $v$ in $ X$  is denoted by $d_{X}(v):=|N_{X}(v)|$.
For a vertex $v\in V$, $N(v)$ denotes the set of neighbors that $v$ has
in $V$, \emph{i.e.}, $N(v):=\{u\in V \; | \; u\sim v\}$; and $N[v]$ denotes
the closed neighborhood of the vertex $v$, \emph{i.e.}, $N[v]:= N(v)\cup \{v\}$.
We denote by $d_G(v_i):=|N(v_i)|$ the degree of a
vertex $v_i\in V$ in $G$, and  by $\d(G),\D(G)$ the minimum and maximum degree of $G$, respectively.
The complement of  a nonempty set $S\subseteq V$ is
denoted by $\overline{S}$ and the complement of $G$ by $\overline{G}$.
The subgraph induced by $S$ is denoted by  $\langle S\rangle$.
As usual, $\mathbbm{1}_X$ denotes the \emph{indicator function} of $X$, \emph{i.e.}, $\mathbbm{1}_X(x)=1$ if $x\in X$ and $\mathbbm{1}_X(x)=0$ otherwise.

In this paper we characterize precisely, for each value of $k$, which graphs have a regular $k$-token graph and which connected graphs have a planar $k$-token graph;
in particular we show the following.

\begin{theorem}\label{thm:reg}
 Let $G$ be a graph of $n$ vertices and let $2 \le k \le n-2$ be an integer. Then $F_k(G)$ is regular
 if and only if one of the following four cases holds.
 \begin{enumerate}
  \item $G$ is isomorphic to the complete graph $K_n$ on $n$ vertices;

  \item $G$ is isomorphic to $\overline{K_n}$;

  \item $G$ is isomorphic  to complete bipartite graph $K_{1,n-1}$ and $k=n/2$;

  \item $G$ is isomorphic  to $\overline{K_{1,n-1}}$ and $k=n/2$.
 \end{enumerate}

\end{theorem}

\begin{theorem}\label{cor:tree}
Let $G$ be a connected graph of order $n>10$ and let $2 \le k \le n-2$ be an integer. Then $F_k(G)$ is planar if and only if
$k=2$ or $k=n-2$, and $G \simeq P_n$.
\end{theorem}

We study regularity in Section~\ref{S_reg} and planarity in Section~\ref{S_Planar}. In
Section~\ref{sec:small} we consider the planarity of the token graphs of graphs
of small order.

\section{Regularity}\label{S_reg}

In this section we prove Theorem~\ref{thm:reg}. We split the proof in two cases: whether $G$ is regular
or not. This are shown in Theorems~\ref{th:Reg} and ~\ref{main-nonregular}, respectively

Let $\overline{F_k(G)}^J$ be the complement of $F_k(G)$ with respect to the Johnson graph, \emph{i.e.},
\[V\big(\overline{F_k(G)}^J\big):=V\big(F_k(G)\big) \text{ and } E\big(\overline{
F_k(G)}^J\big):=E\big(J(n,k)\big)\setminus E\big(F_k(G)\big).
\] The following statement follows easily from the definitions.

\begin{proposition}\label{p:complement}
$F_k\big(\overline{G}\big)=\overline{F_k(G)}^J$ for every graph $G$.
\end{proposition}

Since the Johnson graph $J(n,k)$ is a regular graph, Proposition \ref{p:complement} has the following direct consequence.

\begin{corollary}\label{cor-nonregular}
Let $G$ be a graph such that $F_k(G)$ is regular; then $F_k\big(\overline{G}\big)$ is also regular.
\end{corollary}


\subsection{Regular token graph of a regular graph}

In this section we answer the following question: when is the token graph of a regular graph also regular? We show in Theorem~\ref{th:Reg}
that there are exactly two regular graphs which produce regular token graphs.

\begin{lemma}\label{l:Reg}
Let $G$ be a regular graph and $2\le k\le n-2$ such that $F_k(G)$ is regular.
Then there exist a constant $c$, depending on $k$, the degree of $G$ and the degree of $F_k(G)$,
such that $d_{A}(b)=c$ for every $A\in V\big(F_k(G)\big)$ and every $b\notin A$.
\end{lemma}

\begin{proof}
Fix $A\in V\big(F_k(G)\big)$ and $b\notin A$. Let $B \in V\big(F_k(G)\big)$ with $B\setminus A =\{b\}$
and let $\{a\}=A\setminus B$.
Let $r_1$ and $r_2$ be the degrees of $G$ and $F_k(G)$, respectively.

We first claim that $d_{A\cap B}(b)=d_{A\cap B}(a)$.
Note that
\[
\begin{aligned}
d_{F_k(G)}(A) &= \displaystyle\sum_{u\in A} d_{\overline A}(u) = \sum_{u\in A} \left(r_1 - d_{A}(u) \right) \\
&= kr_1  - \sum_{u\in A} \left( d_{A\cap B}(u) + \mathbbm{1}_{N(a)} (u) \right) \\
&= kr_1  - \sum_{u\in A\cap B} d_{A\cap B}(u) - 2 d_{A\cap B}(a).
\end{aligned}
\]
Analogously, we obtain
\[
d_{F_k(G)}(B) =k r_1  - \sum_{u\in A\cap B} d_{A\cap B}(u) - 2 d_{A\cap B}(b).
\]
Since $F_k(G)$ is regular, the claim follows.

For every $u\in A$, let $A_u:=A\setminus \{u\}$; by the claim we have that $d_{A_u}(u)=d_{A_u}(b)$.
Note that $d_{A}(u)=d_{A_u}(u)$ for every $u\in A$, and that $d_{A_u}(b)=d_A(b)-\mathbbm{1}_{N(b)}(u)$.
Thus, we have $d_{A}(u)=d_{A}(b)-\mathbbm{1}_{N(b)}(u)$ for every $u\in A$. Furthermore, we have $d_A(b)$ vertices $u\in A$ with $d_{A}(u)=d_{A}(b)-1$ and $k-d_A(b)$ vertices $u\in A$ with $d_{A}(u)=d_{A}(b)$.
Therefore,
\[
\begin{aligned}
r_2&=d_A(b)\big(r_1-d_A(b)+1\big)+ \big(k-d_A(b)\big) \big(r_1-d_A(b)\big) \\
&= k\big(r_1-d_A(b)\big) + d_A(b)= kr_1 - (k-1)d_A(b).
\end{aligned}
\]

The result follows with $c:=(r_2-kr_1)/(1-k)$.
\end{proof}

It is a simple fact that $F_k(G)$ is a regular graph for every admissible $k$ if $G$ is any empty graph $E_n$ or any complete graph $K_n$.
The following result shows that they are the unique regular graphs with regular $k$-token graph for  $2\le k\le n-2$.

\begin{theorem}\label{th:Reg}
Let $G$ be regular graph not isomorphic to  either $E_n$ or $K_n$. Then its token graph $F_k(G)$, with $2\le k\le n-2$, is non-regular.
\end{theorem}

\begin{proof}
Suppose to the contrary that $F_k(G)$ is regular. Since $G$ is not isomorphic to $K_n$ nor to
$E_n$, there exists a vertex $v$ of $G$ such that the following holds. Vertex $v$
is of degree at least one and there exists another vertex $v'$ not adjacent to $v$.
Let $A$ be a vertex of $F_k(G)$ such that $A\cap N(v)\neq\emptyset$ and $v,v'\notin A$.
Consider any vertex $u\in A\cap N(v)$ and let $A_u:=\big(A\cup\{v'\}\big)\setminus\{u\}$.
Hence, we have that $d_{A}(v)=d_{A_u}(v)+1$ which contradicts Lemma \ref{l:Reg}.
\end{proof}


\subsection{Regular token graphs of non-regular graphs}

In this section we show that there are exactly two non-regular graphs which produce regular token graphs.
Throughout this subsection, $G$ is a fixed non-regular graph, $u$ and $v$ are vertices of $G$ such that $d_G(u)<d_G(v)$,
and $F:=F_k(G)$. Also, we partition $R:=V(G)\setminus \{u,v\}$ into  four subsets:

\begin{equation*}
\begin{split}
X&:=N(u)\setminus N[v], \\
Y&:=N(v)\setminus N[u], \\
W&:=N(u)\cap N(v), \\
Z&:=R\setminus (X\cup Y \cup W).
\end{split}
\end{equation*}

Note that $Z$ consists precisely of the vertices of $R$ which are nonadjacent to neither $u$ nor $v$ (see Figure~\ref{non-regular}).

\begin{figure}[h]
  \begin{center}
   \includegraphics[width=0.45\textwidth]{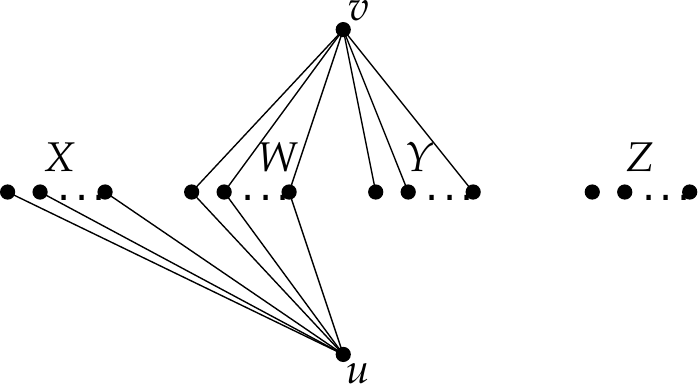}
\end{center}
\caption{Partition of $R$ depending on whether or not $w \in R$ is adjacent to $u$ or $v$.}\label{non-regular}
\end{figure}


In this context, we have the following statements.

\begin{lemma}\label{non-regular1}
If $k-1\leq |X|+|W|+|Z|$, then $F_k(G)$ is non-regular for $2\leq k\leq n-2$.
\end{lemma}

\begin{proof}   We analyze three cases separately.

\begin{itemize}
\item $k-1\leq |Z|$.

We choose a fixed $S\subseteq Z$ such that $|S|=k-1$.  Now consider the vertices $A, B$ of $F$ defined as follows: $A:=S\cup\{u\}$ and $B:=S\cup\{v\}$.
By definition of $F$ we know that the degree of the vertex $A$ (respectively, $B$) in $F$ corresponds to the number of edges of $G$ with one end in $A$ (respectively, $B$)
and the other in $V(G)\setminus A$ (respectively, $V(G)\setminus B$). Let $r$ be the number of edges of $G$ with one end in $S$ and the other in $R\setminus S$.
Then $d_F(A)=d_G(u)+r$ and $d_F(B)=d_G(v)+r$. Since $d_G(v)-d_G(u)>0$, we have $d_F(B)\not=d_F(A)$. Hence, $F$ is not regular.

\item $|Z|<k-1\leq |Z|+|X|$.

We choose a fixed nonempty subset
$S$ of $X$ such that
$|S|+|Z|=k-1$. Thus the sets $A:=S\cup Z\cup\{u\}$ and $B:=S\cup Z \cup\{v\}$ are vertices of $F$.
Let $r$ denote the number of edges of $G$ with one end in $S\cup Z$ and the other in $R\setminus (S\cup Z)$.
Clearly, $d_F(A)=d_G(u) -|S|+r$ and $d_F(B)=d_G(v)+|S|+r$ (recall that $u, v \not \in R$).
Thus $d_F(B)-d_F(A)=d_G(v)-d_G(u)+2|S|>0$, because $d_G(v)-d_G(u)>0$. Hence, $F$ is not regular.

\item $|Z|+|X|<k-1\leq |Z|+|X|+|W|$.

We choose a fixed nonempty subset $S$ of $W$ such that  $|S|+|Z|+|X|=k-1$. Thus the sets $A:=S\cup Z\cup X\cup\{u\}$ and
$B:=S\cup Z \cup X \cup\{v\}$ are vertices of $F$. Let $r$ denote the number of edges of $G$ with one end in $S\cup Z \cup X$ and the other in
$R\setminus (S\cup Z \cup X)$. It is easy to see that $d_F(A)=d_G(u) - |X| + r$ and
$d_F(B)=d_G(v) + |X|  + r$. Thus $d_F(B)-d_F(A)=d_G(v)-d_G(u)+2|X|>0$, because $d_G(v)-d_G(u)>0$. Hence, $F$ is not regular.
\end{itemize}

\end{proof}

\begin{lemma}\label{non-regular2}
If $F_k(G)$ is a regular graph for some $2\leq k\leq n-2$, then $v$ is adjacent to every vertex in $V(G)\setminus \{u,v\}$ in $G$.
\end{lemma}

\begin{proof}

Since $F_{n-k}(G)$ is isomorphic to $F_{k}(G)$, we may assume  that  $2\leq k\leq n/2$.
By Lemma~\ref{non-regular1} and the hypothesis we have that $|X|+|W|+|Z|< k-1\leq n/2 -1$.
Then $|Y|\geq n/2$ because $|X|+|Z|+|W|+|Y|=n-2$.
Suppose that there exists a vertex $y$ in $V(G)\setminus \{u,v\}$ that is nonadjacent to $v$.
Note that $y$ must be an element of $X\cup Z$.

Let $S_1$ be a fixed subset of  $Y$ such that $|S_1|=k-1$. Thus the sets $A_1:=S_1\cup \{u\}$ and  $B_1:=S_1\cup \{v\}$ are vertices of $F$. Let $r_1$ be the number of edges of $G$ with one end in $S_1$ and the other in  $R\setminus S_1$. Then $d_F(A_1)=d_G(u)+(k-1)+r_1$ and
$d_F(B_1)=d_G(v)-(k-1)+r_1$. By the regularity of $F$ we have
$0=d_F(B_1)-d_F(A_1)=d_G(v)-d_G(u)-2(k-1)$, or equivalently

\begin{equation}\label{eq_one}
d_G(v)-d_G(u)=2(k-1).
\end{equation}


Similarly, let $S_2$ be a fixed subset of $Y$ such that $|S_2|=k-2$. Thus the sets $A_2:=S_2\cup \{u,y\}$ and  $B_2:=S_2\cup \{v,y\}$ are vertices of $F$. Let $r_2$ be the number of edges of $G$ with one end in $S_2\cup \{y\}$ and the other in  $R\setminus (S_2\cup \{y\})$.  Then $d_F(A_2)=d_G(u)-\mathbbm{1}_{N(u)}(y)+(k-2)+r_2$ and $d_F(B_2)=d_G(v)-(k-2)+\mathbbm{1}_{N(y)}(u)+r_2.$ Since $\mathbbm{1}_{N(u)}(y)=\mathbbm{1}_{N(y)}(u)$, we have that
$0=d_F(B_2)-d_F(A_2)=d_G(v)- d_G(u)  -2(k-2)+2\mathbbm{1}_{N(y)}(u)$, or equivalently

\begin{equation}\label{eq_two}
d_G(v)- d_G(u)=2(k-2)-2\mathbbm{1}_{N(y)}(u).
\end{equation}


From~\ref{eq_one} and~\ref{eq_two} we obtain

\begin{equation*}
2(k-1)=2(k-2)-2\mathbbm{1}_{N(y)}(u).
\end{equation*}
Which implies that $\mathbbm{1}_{N(y)}(u)=-1$, a contradiction.
\end{proof}

The next  statement is an immediate consequence of Lemma~\ref{non-regular2}.

\begin{corollary}\label{cor:empty}
If $F_k(G)$ is a regular graph, for some $2\leq k\leq n-2$, then  $X$ and $Z$ are empty sets.
\end{corollary}


\begin{lemma}\label{non-regular3}
If $F_k(G)$ is a regular graph for some $2\leq k\leq n-2$, then $d_R(u)=0$.
\end{lemma}

\begin{proof} Again, we may assume  that  $2\leq k\leq n/2$ because $F_{n-k}(G)$ is isomorphic to $F_{k}(G)$.
Suppose to the contrary that $d_R(u)>0$. By Corollary~\ref{cor:empty}, $X$ and $Z$ are empty; therefore, $|W|\geq 1$.

First suppose that $|W|< k-1$. Let $S$ be a  fixed subset of $Y$ such that $|S|+|W|=k-1$.
Then the sets $A:=S\cup W\cup \{u\}$ and $B:=S\cup W\cup \{v\}$ are vertices of F. Let $r$ be the number of edges of $G$ with one end in $S\cup W$ and the other in  $R\setminus (S\cup W)$. Then $d_F(A)=(k-1)+\mathbbm{1}_{N(u)}(v)+r$ and
 $d_F(B)=|W|+(n-2+\mathbbm{1}_{N(v)}(u))-(k-1)+r$. In the last equation we are using that $d_G(v)=n-2+\mathbbm{1}_{N(v)}(u)$  (Lemma~\ref{non-regular2}) . Since $\mathbbm{1}_{N(u)}(v)=\mathbbm{1}_{N(v)}(u)$, we have that $0=d_F(B)-d_F(A)=|W|+(n-2)-2(k-1)$, or equivalently,
$|W|=2k-n$. This implies that $|W|\leq 0$, a contradiction.

Now suppose that $|W|\geq k-1$. Let $S$ be a  fixed subset of $W$ such that
$|S|=k-1$.  Then the sets $A:=S\cup \{u\}$ and $B:=S\cup \{v\}$ are vertices of F. Let $r$ be the number of edges of $G$ with one end in $S$ and the other in  $R\setminus S$. Then $d_G(A)=|W|+\mathbbm{1}_{N(u)}(v)+r$ and $d_G(B)=(n-2)+\mathbbm{1}_{N(v)}(u)+r$.
 Since $\mathbbm{1}_{N(u)}(v)=\mathbbm{1}_{N(v)}(u)$, we have that
$0=d_G(B)-d_G(A)= (n-2) - |W|$, or equivalently, $|W|=n-2$. This implies, $d_G(v)=d_G(u)$, a contradiction.
\end{proof}

We are ready to prove the main result of this section.

\begin{theorem}\label{main-nonregular}
Let $G$ be a non-regular graph such that  $F_k(G)$ is regular for some $2\leq k\leq n-2$.
Then $G$ is isomorphic to $K_{1,n-1}$ or to $\overline{K_{1,n-1}}$, and in both cases $k=n/2$.
\end{theorem}

\begin{proof}
Again, we assume  that  $2\leq k\leq n/2$. First, we show that there are only two possible degrees in $G$. Suppose to the contrary
that there exists three vertices $x_1, x_2$ and $x_3$ such that $d_G(x_1) < d_G(x_2) < d_G(x_3)$. Apply Lemma~\ref{non-regular3} twice:
once with $u=x_1$ and $v=x_3$, and a second time with $u=x_2$ and $v=x_3$. Then $d(x_1)=0$ and $d(x_2)=1$. We obtain a contradiction
by applying Lemma~\ref{non-regular2} with $u=x_2$ and $v=x_3$, since $x_3$ is not adjacent to $x_1$.

Let $r_1$, $r_2$, with $r_1 < r_2$, be the only two possible degrees in $G$. By Lemma~\ref{non-regular3} we have that $r_1=0$ or $r_1=1$.
Let $x$ and $y$ be vertices of $G$ of  degree $r_1$ and $r_2$, respectively.

Suppose that $r_1=0$. We claim that $x$ is the only vertex of degree $0$. Suppose that there exists a second
vertex $z$ with degree $0$. We arrive at a contradiction by applying Lemma~\ref{non-regular2} with $u=x$ and $v=y$, as
$y$ and $z$ are not adjacent. Therefore, all vertices of $G$ distinct from $x$ have degree $r_2$. Moreover,
Lemma~\ref{non-regular2} with $u=x$ and $y=v$  implies that $r_2=n-2$. Therefore, $G$ is isomorphic to $\overline{K_{1,n-1}}$.
Now we show that $k=n/2$.
Let $S$ be any subset of $V(G)\setminus \{x\}$ with $|S|=k-1$.
Then the sets $A:=S \cup \{x\}$ and $B:=S\cup\{z\}$, where $z$ is any element in $V(G)\setminus A$, are vertices of $F$ with $d_F(A)=(k-1)(n-k)$ and $d_F(B)=k(n-k-1)$.
As $F$ is a regular graph we have that $0=d_F(B)-d_F(A)=n-2k$ which implies that $k=n/2$.

Suppose that $r_1=1$. Then, Lemma~\ref{non-regular2} with $u=x$ and $y=v$ implies $r_2=n-1$ and $y$ is adjacent
to every vertex in $G$. Therefore,
there cannot be another vertex distinct from $y$ adjacent to every vertex in $G$ since $d(x)=1$. Thus, $G$
is isomorphic to $K_{1,n-1}$. Finally we show that $k=n/2$.
Let $S$ be any subset of $N(y)$ with $|S|=k-1$. Let $A:=S\cup \{y\}$ and $B:=S \cup\{z\}$, where $z$ is
any element in $V(G) \setminus A$. Then $A$ and $B$ are vertices of $F$ with  $d_F(A)=n-k$ and $d_F(B)=k$.
As $F$ is a regular graph we have that $0=d_F(A)-d_F(B)=n-2k$ which implies that $k=n/2$.

\end{proof}


\section{Planarity}
\label{S_Planar}
In this section we fully characterize, in terms of $G$, when the $k$-token graph of $G$ is planar.
Since $F_1(G)\simeq F_{n-1}(G)\simeq G$, we have that $F_1(G)$ and $F_{n-1}(G)$ are planar if and only if $G$ is planar;
therefore, we only consider the cases when $2\le k\le n-2$ and $n \ge 4$.

As usual, we denote by $G/e$ the graph obtained from  graph $G$ by contracting the edge $e$ of $G$;
and also, we denote by $G-e$ ($G-v$, respectively)
the graph obtained from graph $G$ by deleting the edge $e$ (vertex $v$, respectively) of $G$.
A graph $H$ is a \emph{minor} of a graph $G$ if a graph isomorphic to $H$ can be obtained from $G$ by
contracting some edges, deleting some edges,
and deleting some isolated vertices.  A graph $H$ is a \emph{subdivision} of a graph $G$
if $H$ can be obtained from $G$ by subdividing some edges.

Kuratowski~\cite{kur} proved that a graph is planar if and only if it does not contain
a subdivision of the complete graph $K_5$ nor a subdivision of the bipartite graph $K_{3,3}$.
Wagner~\cite{kur} proved that a graph is planar if and only if it does not contain the complete graph $K_5$ nor the complete bipartite graph $K_{3,3}$ as a minor.

First we show that if $G'$ is a minor of $G$ then $F_{k}(G')$ is a minor of $F_k(G)$.
This result, which is of independent interest, is used to prove the main theorems of this section.

\begin{lemma}\label{l:minor}
If $G'$ is a minor of $G$ then $F_{k}(G')$ is a minor of $F_k(G)$.
\end{lemma}
\begin{proof}
First suppose that  $G'$ is obtained from $G$ from applying one minor operation on $G$.
That is by  deleting a vertex, deleting an edge or contracting an edge;
we show that $F_k(G')$ is a minor of $F_k(G)$.
\begin{itemize}
 \item $G'$ is obtained from $G$ by deleting a vertex $a$.

Then $F_k(G')$ is isomorphic to the graph obtained from $F_k(G)$
by deleting all the vertices of $F_k(G)$ which contain $a$. Thus, $F_k(G')$ is a minor of $F_k(G)$.

\item $G'$ is obtained from $G$ by deleting an edge $[a,b]$.

Then $F_k(G')$ is isomorphic to the graph obtained from $F_k(G)$ by deleting all
the edges $[A,B]$ of $F_k(G)$ such that $A\triangle B=\{a,b\}$. Thus, $F_k(G')$ is a minor of $F_k(G)$.

\item $G'$ is obtained from $G$ by contracting an edge $e:=[a,b]$.

Consider the following subsets of  $V(F_k(G))$:

\begin{eqnarray*}
\mathcal{A}&:=&\left\{A\in V\big(F_k(G)\big) \; \big| \; a\in A, b\notin A \right\},\\
\mathcal{B}&:=&\left\{B\in V\big(F_k(G)\big) \; \big| \; b\in B, a\notin B \right\},\\
\end{eqnarray*}
Clearly, $\mathcal{A}, \mathcal{B}$ are disjoint, and  $|\mathcal A|=|\mathcal B|=\binom{n-2}{k-1}$.
Since $a\sim b$ we have that for each $A\in \mathcal A$ there is exactly one $B\in \mathcal B$ with $A\sim B$ in $F_k(G)$,
in fact $A\triangle B = \{a,b\}$.
Hence, by contracting these $\binom{n-2}{k-1}$ edges from $F_k(G)$,
we obtain a subgraph isomorphic to $F_k(G/e)=F_k(G')$.  Thus, $F_k(G')$ is a minor of $F_k(G)$.
\end{itemize}

Now suppose that $G'$ is obtained by applying two or more minor operations on $G$.
Since the minor relation is a transitive relation, the result follows by induction on the number of these operations.
\end{proof}

Lemma \ref{l:minor}  implies the following.

\begin{theorem}\label{cor:minor}
Let $G$ be a graph and let $H$ be a minor of $G$ such that every non-trivial token graph of $H$ is non-planar. Then $F_k(G)$ is non-planar for $2\le k\le n-2$.
\end{theorem}

 Notice that the order of $H$ in Theorem~\ref{cor:minor} is greater than 4 because $F_2(K_4)$ is a planar graph.
The \emph{circumference} $c(G)$ of a graph $G$ is the supremum of the lengths of its cycles,
if $G$ is a tree we define $c(G)=0$.
The following theorem implies the non-planarity of many token graphs.

\begin{theorem}\label{th:degree+cycle} Let $G$ be a graph and $2 \leq k \leq n-2$. If $\Delta(G)$ or $c(G)$ are greater than or equal to $5$ then $F_k(G)$ is non-planar.
\end{theorem}

\begin{proof}
Since $G$ contains as a minor a star graph $K_{1, 5}$ or a cycle graph $C_5$ and by Theorem~\ref{cor:minor}, it suffices to check that
$F_2(K_{1, 5})$, $F_3(K_{1, 5})$ and $F_2(C_5)$ are non-planar. As shown in Figures~\ref{fig:C5} and~\ref{fig:StarsToken},
$F_2(C_5)$ and $F_2(K_{1, 5})$ both contain a subdivision of $K_5$ and
thus are not planar. For the case of $F_3(K_{1, 5})$, note that by contracting the edges joining vertices of the same color in Figure \ref{fig:StarsToken}
 we obtain the complete graph $K_5$. Thus, $F_3(K_{1, 5})$ is non-planar and the proof is completed.
\end{proof}

\begin{figure}[h]
  \begin{center}
   \includegraphics[width=1.0\textwidth]{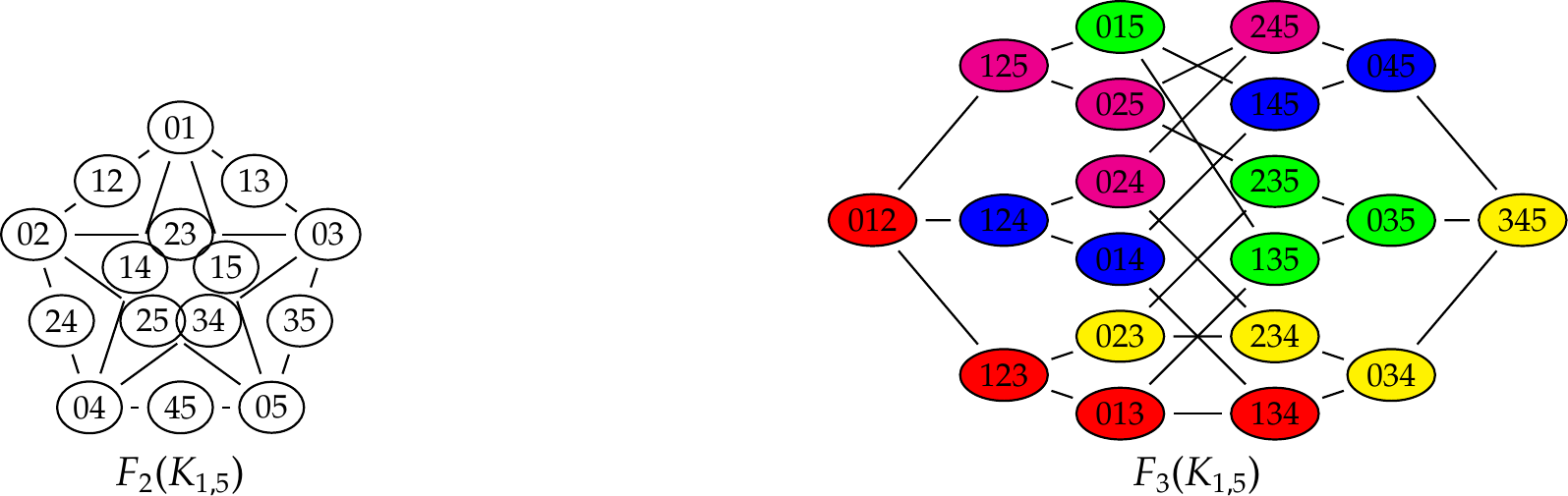}
\end{center}
\caption{The non-planar graphs $F_2(K_{1, 5})$ and $F_3(K_{1, 5})$ (both contain $K_5$ as a minor).} \label{fig:StarsToken}
\end{figure}

The following results shows that others token graphs not included in the previous theorems are non-planar, too.
In particular paths do not verify the hypotheses of previous theorems; however, many of their token graphs are non-planar.

The following two propositions are direct consequences of Theorem 4.1 in \cite{behz} and Theorem 10 in \cite{FFHH}.

\begin{proposition}\label{cor:planar1}
Let $G$ be a graph containing a path $P_3$ on three vertices as a subgraph and
let $2\le k\le n-2$. If $F_{k-1}(G-P_3)$ or $F_{k-2}(G-P_3)$ have maximum degree greater than $2$, then $F_k(G)$ is non-planar.
\end{proposition}

\begin{proposition}\label{cor:planar2}
Let $G$ be a graph and $2\le k\le n-2$. If $G$ contains two disjoint subgraphs isomorphic to $P_3$ and to $K_{1, 3}$ respectively, then $F_k(G)$ is non-planar.
\end{proposition}

We now show that the non-trivial token graphs of all trees (with the exception $P_n$) of more than $10$ vertices  are non-planar.

\begin{theorem}\label{th:tree}
 Let $T$ be a tree of order $n >10 $ non-isomorphic to $P_n$. Then for every $2 \le k \le n-2$, the graph $F_k(T)$ is non-planar.
\end{theorem}

\begin{proof}
Let $2 \le k \le n-2$. Since $T\not\simeq P_n$ we have $\D \geq 3$. The case when $T$ has maximum degree $\D \geq 5$ follows by Theorem \ref{th:degree+cycle}. We consider the cases $\D=4$ and $\D=3$ separately.

Suppose that $\D=4$.
Consider a vertex $v\in V(T)$ with $d(v)=4$. Note that if there exists a vertex $u\in V(T)\setminus\{v\}$ with $d(u)\ge3$ then we have a subgraph $P_3\subset N[u]\setminus\{v\}$,
and so $F_k(T)$ is non-planar by Proposition \ref{cor:planar2}. Now, if $d(u)\le2$, for every $u\in V(T)\setminus\{v\}$, and since $n > 10$, then there are vertices $u_1,u_2,u_3$ in $V(T)$ such that the distance $d(v,u_i)=i$ for $i=1,2,3$ with $\langle \{u_1,u_2,u_3\} \rangle \simeq P_3$.
Therefore, Proposition \ref{cor:planar2} implies that $F_k(T)$ is non-planar.

Suppose that $\D=3$.
Consider a vertex $v\in V(T)$ with $d(v)=3$ and let $N(v)=\{v_1,v_2,v_3\}$.
Suppose first that no other vertex of $T$ has degree $3$.
Since $n > 10$, there is a path $P_3$ in $V(T)\setminus N[v]$ and so Proposition \ref{cor:planar2} gives that $F_k(T)$ is non-planar.
Suppose now that there are at least two vertices in $V(T)$ with degree $3$.
Note that if there is a vertex $u\in V(T)\setminus N[v]$ with $d(u)=3$, then there is a path $P_3$ in $N[u]\setminus N[v]$, and so,
Proposition \ref{cor:planar2} gives that $F_k(T)$ is non-planar.
Thus we can assume that there is no $u\in V(T)\setminus N[v]$ with $d(u)=3$.
Note that if $d(v_1)=d(v_2)=3$ then Proposition \ref{cor:planar2} gives that $F_k(T)$ is non-planar taking a $P_3$ in $N[v_1]\setminus N[v_2]$.
Thus, we can assume without loss of generality that $d(v)=d(v_1)=3$ and $d(u)\le2$ for $u\in V(T)\setminus\{v,v_1\}$.
Since $n > 10$, there exists a $P_3$ in $\overline{N[v]\cup N[v_1]}$, and so, Proposition \ref{cor:planar2} gives that $F_k(T)$ is non-planar.
\end{proof}

Note that the $2$-token graph $F_2(P_n)$ of every path graph $P_n$ with $n$ vertices is planar, see \cite[Figure 1]{FFHH}.
However, the following result shows that the token graph of $P_n$ is non-planar for $3 \le k\le n-3$ and $n\ge 7$.

\begin{theorem}\label{th:path}
Let $G$ be a graph and $3\le k\le n-3$. If $G$ contains a path with $7$ vertices, then $F_k(G)$ is non-planar.
\end{theorem}

\begin{proof}
Let $P_7:=(v_1,\ldots,v_7)$ be a path of seven vertices in $G$. We first show that $F_k(G)$ contains
$F_3(P_7)$ as a subgraph. This follows immediately if $k=3$. Thus, assume that $k\ge 4$.
Fix $k-4$ tokens at vertices in $\overline{V(P_7)}$ and let $H$ be the subgraph
of $F_k(G)$ that results from moving the remaining four
tokens freely. Then, $H$ contains $F_4(P_7)$ as a subgraph, but $F_4(P_7) \simeq F_3(P_7)$.
Thus $F_3(P_7)$ is a subgraph of $F_k(G)$ as claimed. Now, by Proposition \ref{cor:planar1} and the fact
that maximum degree of $F_2(P_4)$ is three, $F_3(P_7)$ is non-planar.
The result follows.
\end{proof}

With Theorems \ref{th:degree+cycle}, \ref{th:tree} and \ref{th:path} we are ready to prove Theorem~\ref{cor:tree}.

\begin{customthm}{2}
Let $G$ be a connected graph of order $n>10$ and let $2 \le k \le n-2$ be an integer. Then $F_k(G)$ is planar if and only if
$k=2$ or $k=n-2$, and $G \simeq P_n$.
\end{customthm}

\begin{proof}
The proof follows from the fact that any connected graph which is non-isomorphic to neither $P_n$ nor $C_n$ has an spanning tree
non-isomorphic to $P_n$. Thus, Theorem \ref{th:tree} gives the result if $G\not\simeq C_n$. However, if $G\simeq C_n$
then Theorem \ref{th:degree+cycle} gives the result.
\end{proof}


\subsection{Graphs of small order}\label{sec:small}

The only one non-trivial token graph of $K_4$ is $F_2(K_4)$, it is isomorphic to the octahedral graph, which is planar. Therefore, all token graphs
 of graphs with at most four vertices are planar.
 Theorem~\ref{cor:tree} implies that all the non-trivial token graphs of a connected graph not isomorphic $P_n$ of more than
 ten vertices are non-planar.
 Thus, the graphs of more than four and at most ten vertices, whose non-trivial token graphs are planar, remain to be found.
 Note that if $G'$ is a subgraph of $G$ then $F_k(G')$ is a subgraph of $F_k(G)$.
 Therefore, it is sufficient to search for the connected graphs edge-maximal with the property that
 their $k$-token graphs are planar.

 For each value of $k$, we did an exhaustive search for these graphs as follows. Since
 $F_k(G)\simeq F_{n-k}(G)$, we considered only those graphs of order at least $2k$.
 Using \texttt{nauty}~\cite{nauty} we generated every connected graph on $n \ge 2k$ vertices
 and $m$ edges. We started our search at $m=n-1$; afterwards, we increased $m$ by one. We stopped as soon
 as all graphs of $n$ vertices and $m$ edges have non-planar $k$-token graphs.
 For a given graph $G$ we tested whether its $k$-token graph is planar and
 whether the addition of any new edge to $G$ produces a graph whose $k$-token graph is non-planar.
 To check for planarity we used \texttt{sage}~\cite{sage}, which in turn uses Boyer's
 implementation of~\cite{boyer}. We found $13$ connected graphs edge-maximal with the property that their $2$-token graphs are planar;
 these are shown in Figure~\ref{fig:M_2}.  For $k=3$  we found the two graphs of six vertices shown on Figure~\ref{fig:M_3_6}.
 All $3$-token graphs of connected graphs with seven or more vertices are non-planar. For $k\ge 4$,  all connected
 graphs of $2k$ or more vertices have non-planar $k$-token graphs.

\section*{Acknowledgments}

The authors would like to thank the anonymous referee for his/her useful suggestions. W. C. was partially supported by the Spanish Ministry of Economy and Competitiveness through project MTM2013-46374-P. J. L.
was partially supported by CONACyT Mexico grant 179867 and by European Union and Republic of Slovenia through the grant
"Internationalization as the pillar of development of University of Maribor". L. M. R. was partially supported by PROMEP grant UAZ-CA-169 and PIFI (Mexico).

\begin{figure}
  \begin{center}
   \includegraphics[width=0.9\textwidth]{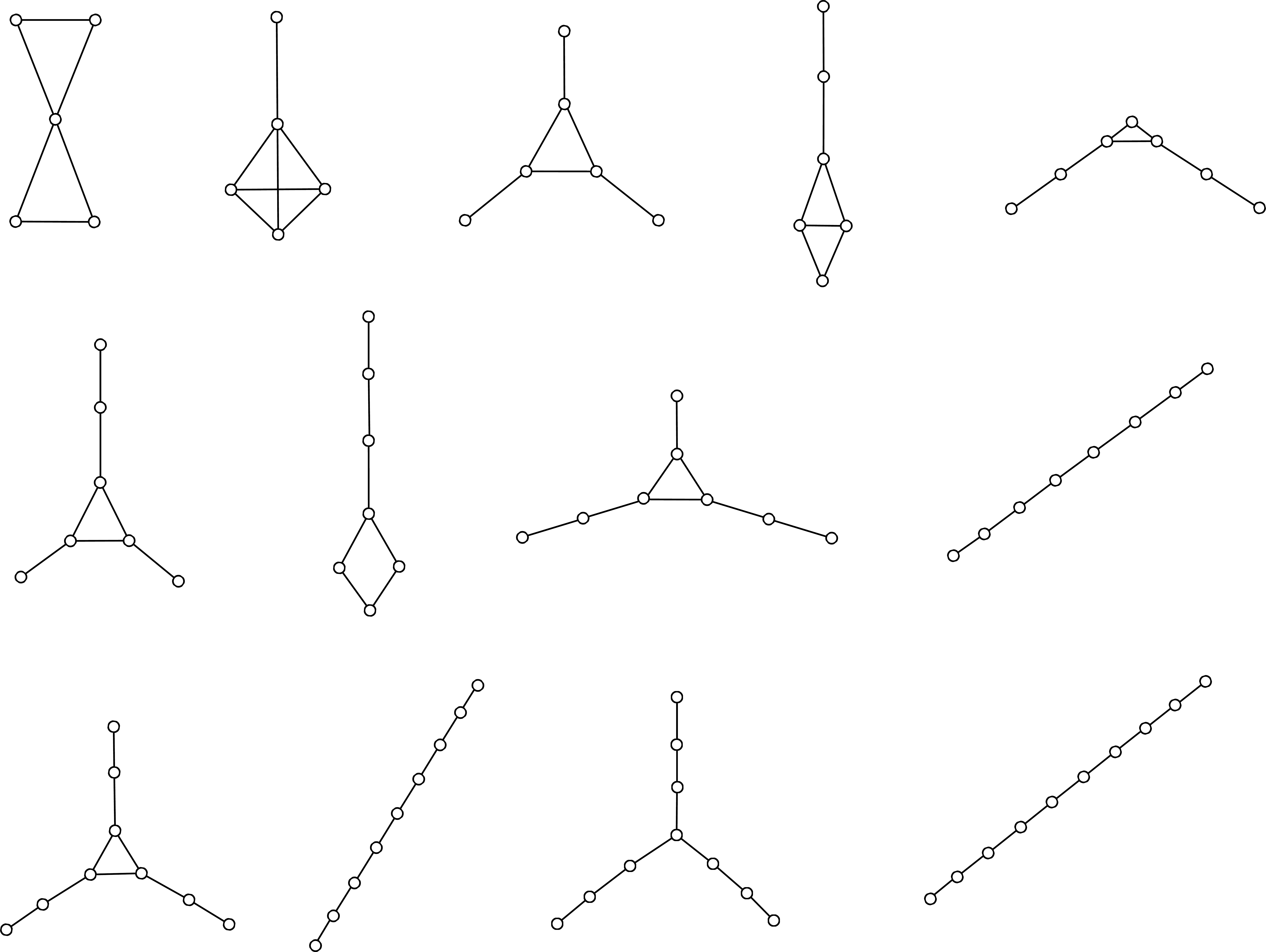}
\end{center}
\caption{The connected edge-maximal graphs on $5,\dots,10$ vertices whose $2$-token graphs are planar.} \label{fig:M_2}
\end{figure}
\begin{figure}
  \begin{center}
   \includegraphics[width=0.8\textwidth]{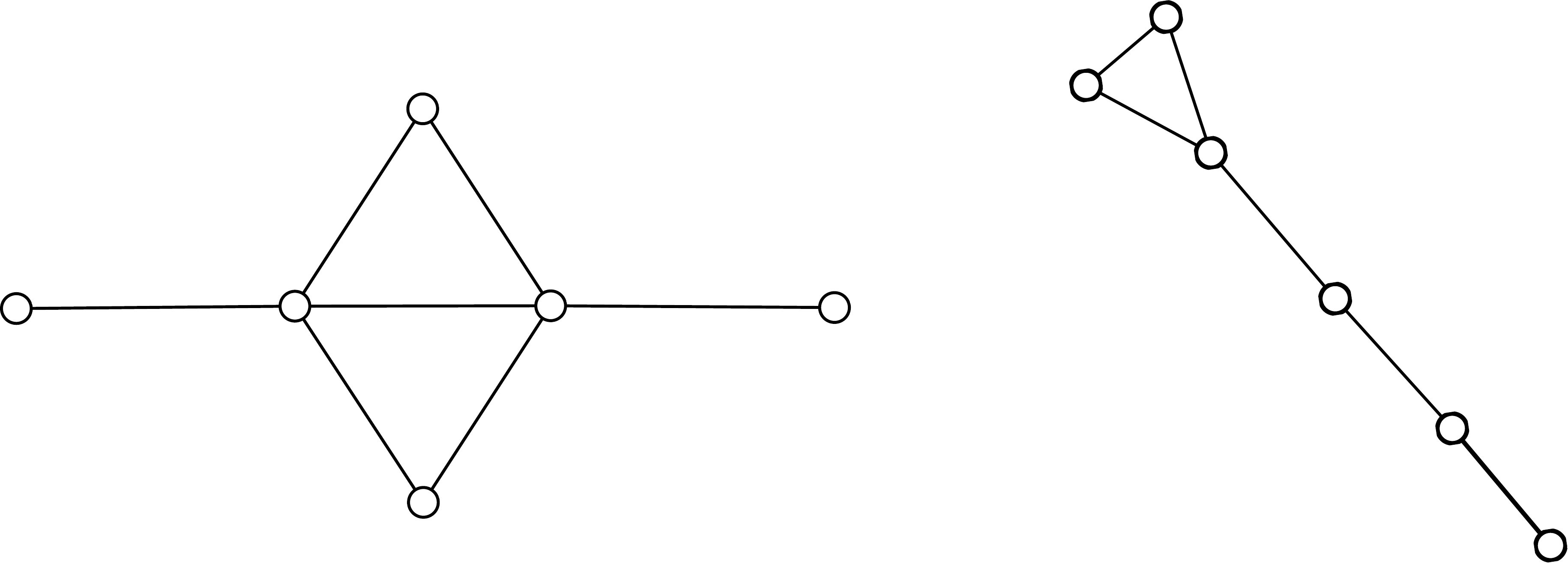}
\end{center}
\caption{The connected edge-maximal graphs on $6$ vertices whose $3$-token graphs are planar.} \label{fig:M_3_6}
\end{figure}

\end{document}